\documentclass[11pt]{amsart}
\usepackage[verbose,letterpaper,tmargin=3cm,bmargin=3cm,lmargin=2.5cm,rmargin=2.5cm]{geometry}
\usepackage{amsmath,amssymb,amsthm,mathrsfs,latexsym, mathtools}
\usepackage[all]{xy}
\usepackage[T1]{fontenc}
\usepackage[utf8]{inputenc}
\usepackage{enumitem}
\usepackage{cite}
\usepackage{txfonts}
\usepackage[usenames, dvipsnames]{color}
\usepackage{listings}
\usepackage{siunitx}
\usepackage[draft=false,linktocpage=true]{hyperref}
\usepackage[capitalise]{cleveref}
\usepackage{comment}
\usepackage{calligra}
\usepackage{setspace}
\usepackage{natbib}
\usepackage[final]{microtype}
\usepackage{tikz,tikz-cd}
\usetikzlibrary{trees,calc, arrows, decorations.markings}
\usepackage{graphicx}
\usepackage{wrapfig}
\usepackage{float}
\usepackage{caption}
\usepackage{subcaption}
\numberwithin{equation}{section}
\newtheorem{theorem}{Theorem}[section]

\newtheorem{proposition}[theorem]{Proposition}
\newtheorem{lemma}[theorem]{Lemma}
\newtheorem{definition}[theorem]{Definition}
\newtheorem{corollary}[theorem]{Corollary}
\newtheorem{question}[theorem]{Question}
\newtheorem{remark}[theorem]{Remark}

\DeclareMathOperator{\Hom}{Hom}
\DeclareMathOperator{\CAlg}{CAlg}

\DeclareMathOperator{\Mod}{Mod}

\DeclareMathOperator{\Cat}{Cat}
\DeclareMathOperator{\fib}{fib}
\DeclareMathOperator{\cofib}{cofib}

\DeclareMathOperator{\iHom}{\mathscr{H}\text{\kern -3pt {\calligra\large om}}\,}
\DeclareMathOperator{\iEnd}{\mathscr{E}\text{\kern -3pt {\calligra\large nd}}\,}
\DeclareMathOperator{\iExt}{\mathscr{E}\text{\kern -3pt {\calligra\large xt}}\,}
\DeclareMathOperator{\iTor}{\mathscr{T}\text{\kern -3pt {\calligra\large or}}\,}

\DeclareMathOperator{\colim}{colim}
\DeclareMathOperator{\laxcolim}{lax\,colim}

\tikzset{%
    symbol/.style={%
        draw=none,
        every to/.append style={%
            edge node={node [sloped, allow upside down, auto=false]{$#1$}}}
    }
}

\newcommand{\Spec}[1]{\mathrm{Spec} \, #1}

\newcommand{\Sp}{\mathrm{Sp}}

\newcommand\opposite[1]{{#1}^{\mathrm{op}}}
\newcommand{\id}{\mathrm{id}}
\newcommand{\QC}[1]{\mathrm{QCoh(#1)}}
\newcommand{\Coh}[1]{\mathrm{Coh(#1)}}

\newcommand{\Perf}[1]{\mathrm{Perf(#1)}}
\newcommand{\Ind}[1]{\mathrm{Ind(#1)}}

\newcommand{\heart}{\ensuremath\heartsuit}

\title{Descendability of Faithfully Flat Covers of Perfect Stacks}
\author{Andy Jiang}
\address{Institute of Mathematics, Academia Sinica, 6F, Astronomy-Mathematics Building, No. 1, Sec. 4, Roosevelt Road, Da-an, Taipei 106319, TAIWAN}
\email{jianga2718@gate.sinica.edu.tw}

\begin{document}

\maketitle

\begin{abstract}
  In \cite{grusonjensen}, Gruson and Jensen gave a new proof of the fact that, over a ring which is either Noetherian of Krull dimension $n$ or
  of cardinality $< \aleph_n$, the projective dimension of any flat module is at most $n$. In this short paper,
  we observe that their arguments apply to the setting of quasicoherent sheaves over perfect stacks.
  As a consequence, we show that for any perfect stack $\mathfrak{X}$ with a faithfully flat cover $p : \Spec{R} \to \mathfrak{X}$
  where $R$ is a Noetherian $\mathbb{E}_{\infty}$-ring of finite Krull dimension or satisfies
  the cardinality bound $2^{|\pi_*(R)|} < \aleph_{\omega}$, $p_*(\mathcal{O}_{\Spec{R}})$ is a descendable algebra in $\QC{\mathfrak{X}}$.
\end{abstract}

\section{Introduction}
After introducing the notion of a descendable algebra in \cite{descendable}, Mathew asked whether
there are examples of faithfully flat morphisms of $\mathbb{E}_{\infty}$-rings which are not descendable.
Recently, this question was settled by \cite{aoki} and \cite{zelich} independently,
who constructed faithfully flat extensions of classical rings which are not descendable.
In this paper, we study the relationship between faithful flatness and descendability of an algebra
in the category of quasicoherent sheaves on a perfect stack. As affine schemes are particular cases
of perfect stacks, we know from the aforementioned works that there is no implication from
faithful flatness to descendability in general. However, we observe that classical arguments of
\cite{grusonjensen} show that, like in the case of affine schemes, any faithfully flat algebra
which is not descendable must be rather \textit{large}.

\begin{theorem}
  Suppose $\mathfrak{X}$ is a quasi-geometric stack with a faithfully flat map
  \[p : \Spec{R} \to \mathfrak{X}\]
  Then $p_*(\mathcal{O}_{\Spec{R}})$ is descendable in any of the following cases
  \begin{enumerate}
  \item $\mathfrak{X}$ is a perfect stack and $R$ is Noetherian of finite Krull dimension.
  \item $\mathfrak{X}$ is a perfect stack and $2^{|\pi_*(R)|} < \aleph_{\omega}$.
  \item $\mathfrak{X}$ has finite cohomological dimension and $R$ is a discrete regular Noetherian ring of finite Krull dimension. \label{option3}
  \end{enumerate}
  In case (\ref{option3}), $\mathfrak{X}$ is automatically a perfect stack.
\end{theorem}
This will be shown in Proposition \ref{coverfflat}, Theorem \ref{mainthm1}, \ref{mainthm1.5}, \ref{mainthm2}, Corollary \ref{autoperfect} and Remark \ref{cardinalremark}.
\vspace{1em}
    
The paper is divided into two parts. The first part is a review of \cite{grusonjensen}, specifically Section 6 of \textit{loc. cit.}, with the arguments rewritten in
the language of idempotent-complete stable categories with a right-bounded t-structure, instead of abelian categories. We make two remarks. One, all idempotent-completeness
assumptions can be dropped with no effect. Two, the change in language from \cite{grusonjensen} make the results slightly more general; however
this additional generality is not essential to the present paper. This is because for any stable category with right-bounded t-structure $\mathscr{C}$,
we can consider the map (see \cite{kobstruct} Proposition 3.26) $D^{b}(\mathscr{C}^{\heart}) \to \mathscr{C}$, which is an isomorphism on the heart.
The $\lambda$-dimension of $\mathscr{C}$ is thus bounded above by the $\lambda$-dimension of $D^{b}(\mathscr{C}^{\heart})$ (see Definition \ref{lambdadef})
and all of our bounds on $\lambda$-dimensions can be deduced from the abelian case, which were established in \cite{grusonjensen}.
Both of these choices reflect stylistic choices of the author and we apologize for any confusion. The reader can safely
skip to Section \ref{section2} if they are familiar with the arguments in \cite{grusonjensen}.

The main arguments of this paper are in Section \ref{section2}. We summarize the strategy here. Suppose $\mathfrak{X}$ is a quasi-geometric stack of finite cohomological dimension
covered by a discrete regular Noetherian ring $R$ of dimension $n$ along the map $p$. The characterization of descendability in \cite{bhattscholze} reduces us to showing
that $\mathcal{O}_{\mathfrak{X}}$ has finite injective dimension. For $x \in \Perf{\mathfrak{X}}^{\heart}$, there is a global bound on the highest nonvanishing cohomology group
of $\Hom_{\QC{\mathfrak{X}}}(x,\mathcal{O}_{\mathfrak{X}})$ coming from the global dimension of $R$ and the cohomological dimension of $\mathfrak{X}$.
This implies, by an argument of Auslander,
that $\mathcal{O}_{\mathfrak{X}}$ has finite injective dimension as $\QC{\mathfrak{X}}^{\heart}$ is locally Noetherian with all Noetherian objects being perfect complexes. In general,
if $\mathfrak{X}$ is covered by a non-regular Noetherian ring, $\mathcal{O}_{\mathfrak{X}}$ may fail to have finite injective dimension.
Therefore, we modify the argument by ``embedding'' (interpreted loosely) $\QC{\mathfrak{X}}$ into a certain functor category, where the image
of $\mathcal{O}_{\mathfrak{X}}$ has finite injective dimension, following \cite{grusonjensen}. This is related to the concept of pure-injective dimensions of flat modules.

\textbf{Conventions}:
We follow the terminology of Lurie in \cite{HTT}, \cite{HA}, and \cite{SAG},
except for referring to an $(\infty,1)$-category simply as a category.
All functors, such as $\Hom$, $\otimes$, $\colim$, and $\lim$ are assumed to be defined in the $\infty$-categorical sense.
Module and quasicoherent sheaves will have no finiteness or boundedness assumptions unless clearly stated.

\textbf{Acknowledgements}:
I am grateful to Wille Liu for explaining a key step in the proof
of Theorem \ref{mainthm1.5}, from which this paper originated.
I thank Bhargav Bhatt, Dmitry Kubrak, Alexander Petrov, and Gleb Terentiuk for useful discussions.
I am thankful to Adeel Khan for his generous support and encouragement, and to Academia Sinica for providing an excellent research environment.
I also benefited from some discussions with the AI model Gemini 2.5 Pro.

\section{Injective Dimension of t-Exact Functors}
In this section we give an exposition of some arguments of Section 6 of \cite{grusonjensen} applied to the slightly more general setting of idempotent-complete
stable categories with a right-bounded t-structure, i.e. such that every object is cohomologically bounded.
This setting is equivalent to that of idempotent-complete prestable categories with finite limits.
This is because given any idempotent-complete prestable category with finite limits $\mathscr{C}$,
the Spanier-Whitehead construction (see \cite{SAG} C.1.1 for a detailed reference) associates to $\mathscr{C}$ a stable category,
\[SW(\mathscr{C}) := \colim (\mathscr{C} \xrightarrow{\Sigma} \mathscr{C} \xrightarrow{\Sigma} \ldots)\]
which carries a right-bounded t-structure whose connective part is $\mathscr{C}$ (see \cite{SAG} Proposition C.1.2.9). 
On the other hand, any idempotent-complete stable category $\mathscr{D}$ with a right-bounded t-structure is reconstructed from its connective part by the same procedure.
The importance of this setup for us is that, for any idempotent-complete stable category $\mathscr{C}$ with a right-bounded t-structure, we have the isomorphisms
\[\Ind{\mathscr{C}} \cong \Ind{SW(\mathscr{C}_{\ge 0})} \cong \Sp(\Ind{\mathscr{C}_{\ge 0}})\]
so that $\Ind{\mathscr{C}}$ is the stablization of a compactly-generated Grothendieck prestable category, see \cite{SAG} Remark C.1.1.6
for an explanation of the second isomorphism.

\begin{definition} \label{lambdadef}
  Let $\mathscr{C}$ be a small idempotent-complete stable category with a right-bounded t-structure.
  We say that $\mathscr{C}$ has $\lambda$-dimension at most $n$ if every t-exact functor (of stable categories)
  \[F : \opposite{\mathscr{C}} \to \Sp\]
  has injective dimension at most $n$ in $\Ind{\mathscr{C}} \cong \mathrm{Fun}^{ex}(\mathscr{C}^{\operatorname{op}},\Sp)$\footnote{See \cite{dualizable} Example 1.5}, i.e. 
  \[\Hom(X,F) \in \mathrm{Sp}^{[0,n]}\]
  for all $X \in \Ind{\mathscr{C}}^{\heart}$. Note that $\Hom(X,F)$ is automatically coconnective.
\end{definition}

The following is a mild generalization of Theorem 6.8 in \cite{grusonjensen}.
\begin{proposition} \label{lambdabound1}
  Suppose $\mathscr{C}_0$, $\mathscr{C}_1$, and $\mathscr{C}_2$ are small idempotent-complete stable categories with right-bounded t-structures such that
  \[\mathscr{C}_0 \xrightarrow{i_1} \mathscr{C}_1 \xrightarrow{p_2} \mathscr{C}_2\]
  is a short exact sequence in $\Cat^{\mathrm{perf}}$ comprising of t-exact functors.
  Then $\mathscr{C}_1$ has $\lambda$-dimension at most $m+n+1$ if $\mathscr{C}_0$ has $\lambda$-dimension at most $m$ and $\mathscr{C}_2$ has $\lambda$-dimension at most $n$.
\end{proposition}
\begin{proof}
  Let \[F : \opposite{\mathscr{C}_1} \to \mathrm{Sp}\]
  be a t-exact functor of stable categories. $F$ is coconnective in $\Ind{\mathscr{C}_1}$
  (since $\Hom(X,F) \in \Sp^{\ge 0}$ for all $X \in (\opposite{\mathscr{C}})_{\ge 0} \cong (\mathscr{C}^{\ge 0})^{\operatorname{op}}$)
  and thus we can (inductively) choose a minimal injective resolution of $F$ in $\Ind{\mathscr{C}_1}$ in the sense that
  \[F \cong \lim_k \mathrm{Fil}^n(F)\]
  where
  \[\mathrm{Fil}^0(F) \cong 0\]
  and for all $k \ge 0$
  \[\fib(F \to \mathrm{Fil}^k(F))[k] \to \fib(\mathrm{Fil}^{k+1}(F) \to \mathrm{Fil}^k(F))[k]\]
  is an injective hull in $\Ind{\mathscr{C}_1}$ in the sense of \cite{SAG} C.5.7.9. It suffices to show that $F \cong \mathrm{Fil}^{m+n+2}(F)$.

  We claim that the minimal injective resolution of $F$ is sent to a minimal injective resolution of $p_1(F)$ under the functor
  \[p_1 : \Ind{\mathscr{C}_1} \to \Ind{\mathscr{C}_0} \]
  which is right adjoint to $\Ind{i_1}$, the extension of $i_1$ to the Ind category.
  It suffices to show that $p_1$ preserves injective hulls. We verify the conditions. $p_1$ preserves injective objects as $\Ind{i_1}$ is t-exact.
  Because $p_1$ preserves coconnective objects (as $\Ind{i_1}$ preserves connective objects),
  it preserves maps of coconnective objects which are monomorphisms on $\pi_0$ (by observing that the cofibre is coconnective). Lastly, we must verify that
  \[H^0(p_1(\_)) : \Ind{\mathscr{C}_1}^{\heart} \to \Ind{\mathscr{C}_0}^{\heart}\]
  preserves essential extensions.

  Suppose $s : M \to N$ is an essential extension in $\Ind{\mathscr{C}_1}^{\heart}$ and $A$ is a subobject of $H^0(p_1(N))$ whose intersection with $H^0(p_1(M))$ is zero. We have the commutative diagram in $\Ind{\mathscr{C}_1}^{\heart}$ (the inclusion functor induced by $i_1$ is suppressed)
  \begin{equation} \label{checkpullback}
    \begin{tikzcd}
      H^0(p_1(M)) \arrow[r] \arrow[d] & H^0(p_1(N)) \arrow[d] \\
      M \arrow[r] & N
    \end{tikzcd}
  \end{equation}
  
  This can be seen as a pullback square as follows. The pullback $H^0(p_1(N)) \times_N M$ is inside the subcategory $\Ind{\mathscr{C}_0}^{\heart}$
  (as it is a subobject of $H^0(p_1(N))$ and the subcategory $\Ind{\mathscr{C}_0}^{\heart}$ is closed under subobjects--a fact which can be seen using the exactness of $\Ind{p_2}$).
  As $H^0(p_1(\_))$ preserves limits, it suffices to check (\ref{checkpullback}) is a pullback square after applying $H^0(p_1(\_))$, where it is clear.
  Therefore, $A$ can be viewed as a subobject of $N$ whose intersection with $M$ is zero--hence it is zero.

  Now because $\mathscr{C}_0$ has $\lambda$-dimension at most $m$ and $p_1(F)$ defines a t-exact functor in $\mathrm{Fun}^{ex}(\opposite{\mathscr{C}_0},\Sp)$,
  we have that $\fib(p_1(F) \to p_1(\mathrm{Fil}^m(F)))[m]$ is an injective object and therefore $p_1(F) \cong p_1(\mathrm{Fil}^{m+1}(F))$ (and same for $p_1(\mathrm{Fil}^{k}(F))$ for $k \ge m+1$).
  We can inductively show that the functors
  \[\fib(F \to \mathrm{Fil}^k(F))[k] : \opposite{\mathscr{C}_1} \to \mathrm{Sp}\] are t-exact for all $k \ge 0$
  by showing that for any $x \in \mathscr{C}^{\heart}_1$, the map
  \[\fib(F \to \mathrm{Fil}^k(F))[k](x) \to \fib(\mathrm{Fil}^{k+1}(F) \to \mathrm{Fil}^k(F))[k](x)\]
  is a monomorphism on $\pi_0$. But this simply follows from the fact that the map
  \[\fib(F \to \mathrm{Fil}^k(F))[k] \to \fib(\mathrm{Fil}^{k+1}(F) \to \mathrm{Fil}^k(F))[k]\]
  is a monomorphism on $\pi_0$. Therefore, for all $k \ge m+1$, $\fib(F \to \mathrm{Fil}^{k}(F))[k]$ is a t-exact functor
  in $\mathrm{Fun}^{ex}(\opposite{\mathscr{C}_1},\mathrm{Sp})$ which vanishes on $\opposite{\mathscr{C}_0}$, therefore they
  define t-exact functors from $\opposite{\mathscr{C}_2}$ to $\mathrm{Sp}$.

  Let 
  \[G : \opposite{\mathscr{C}_2} \to \mathrm{Sp}\]
  be the t-exact functor induced by $\fib(F \to \mathrm{Fil}^{m+1}(F))[m+1]$ and for each $k \ge 0$ let
  \[\mathrm{Fil}^{k}(G) : \opposite{\mathscr{C}_2} \to \mathrm{Sp}\]
  be the functor induced by $\fib(\mathrm{Fil}^{k+m+1}(F) \to \mathrm{Fil}^{m+1}(F))[m+1]$.
  Then
  \[G \cong \lim_{k}{\mathrm{Fil}^k(G)}\]
  We will show that the right hand side defines an injective resolution of $G$ (in $\Ind{\mathscr{C}_2}$).

  First, note that for $k \ge m+1$, $\fib(\mathrm{Fil}^{k+1}(F) \to \mathrm{Fil}^k(F))[k]$ is an injective object in $\Ind{\mathscr{C}_1}$
  which lives in the subcategory $i_2(\Ind{\mathscr{C}_2})$ (where $i_2$ is right adjoint to $\Ind{p_2}$)--since they receive no maps from $\Ind{\mathscr{C}_0}$.
  Therefore for all $k \ge 0$,
  \[\fib(\mathrm{Fil}^{k+1}(G) \to \mathrm{Fil}^{k}(G))[k]\]
  is injective when viewed as an object of $\Ind{\mathscr{C}_2}$ as
  \[i_2(\fib(\mathrm{Fil}^{k+1}(G) \to \mathrm{Fil}^{k}(G))[k]) \cong \fib(\mathrm{Fil}^{k+1}(F) \to \mathrm{Fil}^k(F))[k]\]
  and $\Ind{p_2}$ is exact.
  As $i_2$ reflects coconnective objects,
  \[\fib(G \to \mathrm{Fil}^k(G))[k] \to \fib(\mathrm{Fil}^{k+1}(G) \to \mathrm{Fil}^{k}(G))[k]\]
  is a map of coconnective objects in $\Ind{\mathscr{C}_2}$ which is a monomorphism on $\pi_0$ for all $k \ge 0$.

  Finally we need to show that
  \[\pi_0(\fib(G \to \mathrm{Fil}^k(G))[k]) \to \pi_0(\fib(\mathrm{Fil}^{k+1}(G) \to \mathrm{Fil}^{k}(G))[k])\]
  is an essential extension in $\Ind{\mathscr{C}_2}^{\heart}$.

  Now we have the isomorphisms (as $i_2$ is left t-exact)
  \[H^0(i_2(\pi_0(\fib(G \to \mathrm{Fil}^k(G))[k]))) \cong \pi_0(\fib(F \to \mathrm{Fil}^{k+m+1}(F))[k+m+1])\] 
  and
  \[H^0(i_2(\pi_0(\fib(\mathrm{Fil}^{k+1}(G) \to \mathrm{Fil}^k(G))[k]))) \cong \pi_0(\fib(\mathrm{Fil}^{k+m+2}(F) \to \mathrm{Fil}^{k+m+1}(F))[k+m+1])\] 
  in $\Ind{SW(\mathscr{C}_1)}^{\heart}$. As $H^0(i_2(\_))$ preserves subobjects and pullbacks, a subobject of \[\pi_0(\fib(\mathrm{Fil}^{k+1}(G) \to \mathrm{Fil}^k(G))[k])\]which
  does not intersect \[\pi_0(\fib(G \to \mathrm{Fil}^k(G))[k])\] is sent by $H^0(i_2(\_))$ to a subobject of \[\pi_0(\fib(\mathrm{Fil}^{k+m+2}(F) \to \mathrm{Fil}^{k+m+1}(F))[k+m+1])\]
  which does not intersect \[\pi_0(\fib(F \to \mathrm{Fil}^{k+m+1}(F))[k+m+1])\] Hence it is sent to zero under $H^0(i_2(\_))$ and therefore is zero--as
  \[H^0(i_2(\_)) : \Ind{\mathscr{C}_2}^{\heart} \to \Ind{\mathscr{C}_1}^{\heart}\]
  is fully faithful (because the counit of the adjunction $p_2 \dashv H^0(i_2)$ is $p_2 \circ H^0(i_2) \cong H^0(p_2 \circ i_2) \cong \id$).

  Therefore, we have that $G \cong \mathrm{Fil}^{n+1}(G)$ by the $\lambda$-dimension bound of $\mathscr{C}_2$. Hence $F \cong \mathrm{Fil}^{n+m+2}(F)$ as desired.
  
\end{proof}

We can bound the $\lambda$-dimension of an idempotent-complete stable category $\mathscr{C}$ with a right-bounded t-structure by properties of its heart.
We begin with some well-known lemmas.

%

\begin{lemma} \label{basicabelianlemma}
  Let $\mathscr{A}$ be an abelian category. For any $x \in \mathscr{A}$, we have
  \[\Ind{Sub_{\mathscr{A}}(x)} \cong \mathrm{Sub}_{\Ind{\mathscr{A}}}(x) \]
  where $\mathrm{Sub}_{\mathscr{C}}(x)$ is the poset of subobjects of $x$ in $\mathscr{C}$.
  As $\mathrm{Sub}_{\Ind{\mathscr{A}}}(x)$ has finite colimits, $\Ind{Sub_{\mathscr{A}}(x)}$
  is compactly generated.
\end{lemma}
\begin{proof}
  As the Yoneda inclusion $\mathscr{A} \to \Ind{\mathscr{A}}$ is exact, there is a natural (fully faithful) functor
  \[\mathrm{Sub}_{\mathscr{A}}(x) \to \mathrm{Sub}_{\Ind{\mathscr{A}}}(x)\]
  Because filtered colimits are exact in $\Ind{\mathscr{A}}$, the right hand side has all small filtered colimits
  (which are computed in $\Ind{\mathscr{A}}$).
  This implies the above functor lands in the subcategory of compact objects, $\mathrm{Sub}_{\Ind{\mathscr{A}}}(x)^{\omega}$.
  Thus we can construct a fully faithful filtered colimit preserving functor
  \[\Ind{Sub_{\mathscr{A}}(x)} \to \mathrm{Sub}_{\Ind{\mathscr{A}}}(x) \]
  It suffices to show that it is essentially surjective.

  Suppose $\colim_{i \in I} y_i$ is a subobject of $x$ in $\Ind{\mathscr{A}}$, where the $y_i$'s are in $\mathscr{A}$ and $I$ is filtered.
  Then because filtered colimits are exact in $\Ind{\mathscr{A}}$,
  \[\colim_{i \in I} y_i \cong \mathrm{im}(\colim_{i \in I} y_i \to x) \cong \colim_{i \in I}(\mathrm{im}(y_i \to x))\]
  So we are done because $\mathrm{im}(y_i \to x)$ is a subobject of $x$ in $\mathscr{A}$.
\end{proof}

\begin{lemma} \label{posetlemma}
  Let $P$ be a poset. Then, $\Ind{P}$ is also a poset and there is an isomorphism between $\Ind{P}$
  and the poset of ideals on $P$. This isomorphism takes an ideal $I$ to the colimit of the Yoneda embedding restricted to the ideal.
  Here an ideal of a poset is a downward-closed, filtered subset.
\end{lemma}
\begin{proof}
  Every downward-closed subset of $P$ defines a presheaf on $P$.
  It's easy to see that every representable functor
  is defined by an ideal and the collection of presheaves defined by ideals is closed under filtered colimits.
  Every ideal is a filtered colimit along itself of the Yoneda embedding (by the proof of \cite{HTT} Lemma 5.1.5.3)
  so the result follows.
\end{proof}

The following is a standard fact about compactly generated categories, which is shown at the end
of the proof of \cite{HTT} Theorem 5.5.1.1. 
\begin{proposition} \label{kappacompactlemma}
  Let $\mathscr{C}$ be a compactly generated category and
  $\kappa$ be a regular cardinal. Any $\kappa$-compact object of
  $\mathscr{C}$ is a $\kappa$-small filtered colimit of compact objects.
\end{proposition}
\begin{proof}
  The statement is trivial for $\kappa = \omega$, so assume $\kappa > \omega$.
  Let $X$ be any $\kappa$-compact object of $\mathscr{C}$.
  Using the proof of \cite{HTT} Corollary 4.2.3.11, we can write $X$ as a $\kappa$-filtered
  colimit of $\kappa$-small colimits of compact objects.
  So, we see $X$ is the retract of a $\kappa$-small colimit of compact objects.
  We can then use the proof of \cite{HTT} Corollary 4.2.3.11 to rewrite $X$ as the retract of a $\kappa$-small filtered
  colimit of compact objects, as compact objects are closed under finite colimits.

  Write $X$ as a retract of an idempotent on $Y := \colim_{i \in I} y_i$, where $y_i$ are compact
  and $I$ is a $\kappa$-small filtered category.
  As in the proof of \cite{HTT} Proposition 4.4.5.20, we can write $X$ as a sequential colimit
  along the idempotent endomorphism of $Y$. Let $J$ be the lax sequential colimit
  \[J := \laxcolim{I^{[n]}}\]
  where the transition functor $I^{[n]} \to I^{[n+1]}$ is induced by the map $[n+1] \to [n]$ which sends $x$
  to $\max(x,n)$. We can compute
  \[\colim_{(i_0 \to \ldots \to i_n) \in J} y_{i_n} \cong \colim_{n}\colim_{(i_0 \to \ldots \to i_n) \in I^{[n]}}y_{i_n} \cong \colim_n\colim_{i \in I}y_i \cong X\]
  because the diagonal functor $I \to I^{[n]}$ is cofinal by \cite{HTT} Proposition 5.3.1.22.
  And we are done because $J$ is a $\kappa$-small filtered category.
  
\end{proof}

\begin{definition}
  We say an object $x$ in an abelian category $\mathscr{A}$ is $\kappa$-Noetherian if every filtered subset
  of $\mathrm{Sub}_{\mathscr{A}}(x)$ (the poset of subobjects of $x$ in $\mathscr{A}$)
  has a cofinal subset of cardinality $< \kappa$.
  We say an abelian category $\mathscr{A}$ is $\kappa$-Noetherian if every object in $\mathscr{A}$ is.
  If $x^{\operatorname{op}}$ is $\kappa$-Noetherian in $\opposite{\mathscr{A}}$, we say that $x$ is $\kappa$-Artinian in $\mathscr{A}$.
  $\opposite{\mathscr{A}}$ is called $\kappa$-Artinian if $\mathscr{A}$ is $\kappa$-Noetherian.
  Note that our indexing is off-by-one from \cite{grusonjensen}.
\end{definition}

\begin{lemma} \label{Noetheriansubobject}
  Let $\mathscr{A}$ be an abelian category and $\kappa$ be a regular cardinal. Then for any $x \in \mathscr{A}$, $x$ is $\kappa$-Noetherian in $\mathscr{A}$
  if and only if every subobject of $x$ in $\Ind{\mathscr{A}}$ is $\kappa$-compact.
\end{lemma}
\begin{proof}
  First, we note that filtered colimits in $\mathrm{Sub}_{\Ind{\mathscr{A}}}(x)$ can be computed in $\Ind{\mathscr{A}}$
  by the proof of Lemma \ref{basicabelianlemma}.
  
  Suppose $x$ is $\kappa$-Noetherian in $\mathscr{A}$, and $y$ is a subobject of $x$ in $\Ind{\mathscr{A}}$.
  Because $\mathrm{Sub}_{\Ind{\mathscr{A}}}(x) \cong \Ind{\mathrm{Sub}_{\mathscr{A}}(x)}$,
  we have $y \cong \colim_{i \in I}i$ where $I$ is the ideal of subobjects of $x$ in $\mathscr{A}$ which
  are also subobjects of $y$ (here the colimit can be taken in $\mathrm{Sub}_{\Ind{\mathscr{A}}}(x)$ or $\Ind{\mathscr{A}}$).
  The existence of a $\kappa$-small cofinal subset of $I$ implies $y$ is $\kappa$-compact.
  
  Now assume every subobject of $x$ in $\Ind{\mathscr{A}}$ is $\kappa$-compact and let $I \subseteq \mathrm{Sub}_{\mathscr{A}}(x)$ be a filtered subset.
  Then $\colim_{i \in I}i$ is $\kappa$-compact in $\Ind{\mathscr{A}}$ and therefore the map $\colim_{j \in J}j \to \colim_{i \in I}$
  admits a section for some $\kappa$-small subset $J \subseteq I$ by the proof of \cite{HTT} Corollary 4.2.3.11.
  We can enlarge $J$ into a filtered subset as follows.
  Let $J^{(0)}:=J$. For $n > 0$, define $J^{(n)}$ to be any $\kappa$-small
  subset of $I$ such that for every
  finite subset $K \subseteq J^{(n-1)}$, there is an element in $J^{(n)}$ which is larger or equal to all elements of $K$.
  Let $J^{(\infty)}$ be the union of all the $J^{(n)}$'s, then it is then a $\kappa$-small filtered subset of $I$ such that
  the map $\colim_{j \in J^{(\infty)}}j \to \colim_{i \in I}i$ has a section. Therefore, as this is an inclusion of subobjects,
  it is an isomorphism. Therefore, $J^{(\infty)}$ is cofinal in $I$ and $x$ is $\kappa$-Noetherian.
\end{proof}

\begin{corollary} \label{kappanoetherianext}
  Let $\mathscr{A}$ be an abelian category and $\kappa$ be a regular cardinal. The collection of $\kappa$-noetherian objects
  is a thick subcategory of $\mathscr{A}$.
\end{corollary}
\begin{proof}
  If $x \in \mathscr{A}$ and $y$ is a subobject of $x$, then $\mathrm{Sub}_{\mathscr{A}}(y)$ and $\mathrm{Sub}_{\mathscr{A}}(x/y)$
  embed in $\mathrm{Sub}_{\mathscr{A}}(x)$, so the collection of $\kappa$-Noetherian objects is closed under subobjects
  and quotients by definition. Now suppose \[0 \to a \to b \to c \to 0\] is a short exact sequence in $\mathscr{A}$, with $a$ and $c$
  being $\kappa$-Noetherian. Any subobject of $b$ is an extension of a subobject of $c$ by a subobject of $a$, so it suffices
  to show that $\kappa$-compact objects in $\Ind{\mathscr{A}}$ is closed under extensions by Lemma \ref{Noetheriansubobject}.
  As $\Ind{\mathscr{A}} \cong \Ind{D^b(\mathscr{A})}^{\heart}$ and $\Ind{D^b(\mathscr{A})}$ is a Grothendieck abelian category (\cite{SAG} Example C.1.4.4),
  this follows from the fact that $\kappa$-compact objects in a stable category are closed under finite colimits and shifts (as shifts are isomorphisms
  in a stable category)--together with the fact that $\kappa$-compact objects in $\Ind{\mathscr{A}}$ are $\kappa$-compact
  in $\Ind{D^b(\mathscr{A})}$ (as objects in $A$ are compact in $\Ind{D^b(\mathscr{A})}$ and $\kappa$-compact objects in $\Ind{\mathscr{A}}$
  are $\kappa$-small filtered colimits of objects in $A$).
\end{proof}

The following is a mild generalization of \cite{grusonjensen} Corollary 6.5.
\begin{proposition} \label{lambdabound2}
  Let $\mathscr{C}$ be a small idempotent-complete stable category with a right-bounded t-structure such that $\mathscr{C}^{\heart}$ is
  $\aleph_{k}$-Noetherian, then $\mathscr{C}$ has $\lambda$-dimension at most $k$.
\end{proposition}
\begin{proof}
  Fix a t-exact functor $F : \opposite{\mathscr{C}} \to \Sp$ of stable categories. For any morphism $f: X \to \Sigma^{k+1}F$
  in $\Ind{\mathscr{C}}$ where $X \in \Ind{\mathscr{C}}^{\heart}$,
  we will construct a nullhomotopy of $f$ following the proof strategy of \cite{SAG} Proposition C.6.10.1\footnote{This type of argument goes back to Auslander, see \cite{auslander}}.

  Inductively define a (large) chain of subobjects indexed by ordinal numbers $\{X_{\alpha} \subseteq X\}$ in $\Ind{\mathscr{C}}^{\heart}$ such that
  \begin{itemize}
  \item $X_0 \cong 0$
  \item $X_{\alpha+1}/X_{\alpha}$ admits a surjection from an element of $\mathscr{C}^{\heart}$ and is nonzero if $X_{\alpha} \neq X$
  \item $X_{\alpha} = \cup_{\beta < \alpha}{X_{\beta}}$ if $\alpha$ is a limit ordinal
  \end{itemize}
  This is possible because $\Ind{\mathscr{C}}^{\heart} \cong \Ind{\mathscr{C}^{\heart}}$ is generated by $\mathscr{C}^{\heart}$.
  Note that $X_{\alpha}=X$ for $\alpha \gg 0$.

  We will build the nullhomotopy in $\Hom(X,\Sigma^{k+1}F) \cong \lim_{\alpha}\Hom(X_{\alpha},\Sigma^{k+1}F)$ by transfinite induction.
  It suffices to show that for any successor ordinal $\alpha+1$ and a null homotopy of $f_{\alpha} := f|_{X_{\alpha}}$ in
  $\Hom(X_{\alpha},\Sigma^{k+1}F)$, we can find a compatible nullhomotopy of $f_{\alpha+1}$ in $\Hom(X_{\alpha+1},\Sigma^{k+1}F)$--this will ensure
  we can continue to extend the nullhomotopy at any successor ordinal and the extension of the nullhomotopy to limit ordinals is uniquely defined.

  The map $f_{\alpha+1} : X_{\alpha+1} \to \Sigma^{k+1}F$ together with the nullhomotopy of $f$ restricted to $X_{\alpha}$ defines a point in
  the space
  \[\Hom(X_{\alpha+1},\Sigma^{k+1}F) \times_{\Hom(X_{\alpha},\Sigma^{k+1}F)} 0\]
  So it suffices to show this space is connected. But this space is isomorphic to
  \[\Hom(X_{\alpha}/X_{\alpha+1},\Sigma^{k+1}F)\]
  $X_{\alpha+1}/X_{\alpha}$ is $\kappa$-compact in $\Ind{\mathscr{C}}^{\heart}$ by Lemma \ref{Noetheriansubobject}, so we finish from Lemma
  \ref{Extvanishingonsmallobjects} below.

\end{proof}

\begin{lemma} \label{Extvanishingonsmallobjects}
  Let $\mathscr{C}$ be a small idempotent-complete stable category with a right-bounded t-structure and
  $F : \opposite{\mathscr{C}} \to \Sp$ be a t-exact functor of stable categories.
  For any $\aleph_k$-compact object $X \in \Ind{\mathscr{C}}^{\heart}$, 
  \[\Hom_{\Ind{\mathscr{C}}}(X,F) \in \Sp^{[0,k]}\]
\end{lemma}
\begin{proof}
  $F$ is coconnective because $\Hom(X,F) \in \Sp^{\ge 0}$ for all $X \in \mathscr{C}_{\ge 0}$, so it suffices to show
  \[\Hom_{\Ind{\mathscr{C}}}(X,F) \in \Sp^{\le k}\]
  for all $\aleph_k$-compact $X$ in $\Ind{\mathscr{C}}^{\heart}$.

  $X$ is a $\aleph_k$-small filtered colimit of compact objects in $\Ind{\mathscr{C}}^{\heart}$ (see Proposition \ref{kappacompactlemma}).
  The inclusion $\Ind{\mathscr{C}}^{\heart} \to \Ind{\mathscr{C}}$
  preserves filtered colimits, so the result follows from the fact that $\aleph_k$-small cofiltered limits in $\Sp$ have cohomological dimension at most $k$,
  see \cite{SAG} Lemma D.3.3.6.\footnote{This fact is originally due to \cite{goblot} and the argument given in \cite{SAG} is based on \cite{mitchell}, which was in turn inspired by \cite{osofsky}.}
\end{proof}

Let $R$ be a Noetherian $\mathbb{E}_{\infty}$-ring and $\Coh{R}$ be the stable category of bounded almost perfect $R$-modules.
By \cite{HA} Proposition 7.2.4.18, this category carries a t-structure such that the forgetful functor to $\Sp$ is t-exact.
We can define an increasing filtration on $\Coh{R}$ using dimension of support.
More precisely, let $\Coh{R}_{(k)}$ be the subcategory of modules whose support
(a closed subset of $\Spec{R}$) has Krull dimension at most $k$. The subquotients of this filtration are endowed with induced t-structures by
the Lemma \ref{tstructurelemma} below, since $\Coh{R}^{\heart}_{(k)}:= \Coh{R}_{(k)} \cap \Coh{R}^{\heart}$ is a Serre subcategory of $\Coh{R}^{\heart} \cong \Coh{\pi_0(R)}^{\heart}$
and an object in $\Coh{R}$ is in $\Coh{R}_{(k)}$ if and only if all its homotopy groups are in $\Coh{R}_{(k)}$.

\begin{lemma} \label{tstructurelemma}
  Let $\mathscr{C}$ be a small idempotent-complete stable category with a t-structure. A Serre subcategory $J \subseteq \mathscr{C}^{\heart}$
  defines an idempotent-complete stable subcategory of $\langle J \rangle \subseteq \mathscr{C}$ consisting of objects whose homotopy groups are in $J$.
  There is a unique t-structure on the Verdier quotient $\mathscr{C}/\langle J \rangle$ such that the quotient map is t-exact.
  The heart of this t-structure on $\mathscr{C}/\langle J \rangle$ is $\mathscr{C}^{\heart}/J$.
\end{lemma}
\begin{proof}
  Let $p: \mathscr{C} \to \mathscr{C}/\langle J \rangle$ denote the quotient map.
  We will construct a t-structure on the quotient whose connective and coconnective parts are the images of those of $\mathscr{C}$.
  Then it will automatically be unique (such that $p$ is t-exact) because any other t-structure would have connective and coconnective parts
  containing those of this one.

  To show that the above description defines a t-structure, it suffices to show that the space $\Hom(p(F),p(G))$ is
  contractible if $F \in \mathscr{C}_{\ge 1}$ and $G \in \mathscr{C}^{\ge 0}$.
  We have the computation (see \cite{nikolausscholze} Theorem I.3.3.)
  \[\Hom(p(F),p(G)) \cong \colim_{(a,f) \in \langle J \rangle \downarrow G}{\Hom(F,\cofib(f))}\]
  where the indexing category is filtered because $\langle J \rangle \downarrow G$ has finite colimits computed in $\mathscr{C}$.
  The category $\langle J \rangle \downarrow G$ can be identified with the subcategory of $\mathscr{C}_{G/}$ consisting of objects $g: G \to H$ such
  that $\fib(g) \in \langle J \rangle$. This category has an endofunctor $L$ which takes $g: G \to H$ to $\tau^{\ge 0}(g) : G \to \tau^{\ge 0}H$.
  Note that we have to use the fact that $J$ is a Serre subcategory to deduce that the fibre of $\tau^{\ge 0}(g)$ is in $\langle J \rangle$
  (by studying the associated long exact sequences of homotopy groups).
  We can recognize $L$ as a localization functor by condition (3) of \cite{HTT} Proposition 5.2.7.4 .

  Using \cite{HTT} Theorem 4.1.3.1, we can conclude that the image of $L$ is cofinal, and thus also filtered. So the Hom space
  \[\Hom(p(F),p(G)) \cong \colim_{\{(g,H) \in \mathscr{C}^{\ge 0}_{G/} | \fib(g) \in \langle J \rangle\}}{\Hom(F,H)}\]
  is contractible by \cite{HTT} Lemma 5.3.1.20.

  Finally, let $x$ and $y$ be two element of $\mathscr{C}^{\heart}$. By \cite{HTT} Theorem 4.1.3.1, the functor
  \[\pi_0: {\{(g,z) \in \mathscr{C}^{\ge 0}_{y/} |\, \fib(g) \in \langle J \rangle\}}
    \to  {\{(g,z) \in \mathscr{C}^{\heart}_{y/} |\, \mathrm{ker}(g), \mathrm{coker}(g)\in  J \}} \]
  which takes $(g,z)$ to $(\pi_0(g),\pi_0(z))$ is cofinal (since it is a right adjoint). Hence, we have the isomorphism of spaces
  \[\Hom(p(x),p(y)) \cong \colim_{\{(g,z) \in \mathscr{C}^{\heart}_{y/} |\, \mathrm{ker}(g), \mathrm{coker}(g)\in  J \}}{\Hom_{\mathscr{C}}(x,z)}\]
  which agrees with the mapping set in $\mathscr{C}^{\heart}/J$ by \cite{stacks} Lemma 02MS and Remark 05Q0.
\end{proof}

The following result is a theorem of Gabriel, see \cite{gabriel}.
\begin{theorem} \label{Gabrieltheorem}
  For each $k$, the subquotient $\Coh{R}^{\heart}_{(k+1)}/\Coh{R}^{\heart}_{(k)}$ is a finite-length abelian category with the isomorphism classes of simple objects in
  bijection with primes ideals $\mathfrak{p} \subset \pi_0(R)$ such that $\pi_0(R)/\mathfrak{p}$ is $(k+1)$-dimensional.
\end{theorem}
\begin{proof}
  By \cite{stacks} Lemma 00L0, any object in $\Coh{R}^{\heart}_{(k+1)}$ has a finite
  length filtration with subquotients of the form $\pi_0(R)/\mathfrak{p}$ where $\mathfrak{p}$ is prime and $\pi_0(R)/\mathfrak{p}$ has
  dimension at most $k+1$. So it suffices to show that $\pi_0(R)/\mathfrak{p}$ is simple in $\Coh{R}^{\heart}_{(k+1)}/\Coh{R}^{\heart}_{(k)}$
  if $\pi_0(R)/\mathfrak{p}$ is of dimension exactly $k+1$. This follows from the fact that
  for any submodule $I/\mathfrak{p} \subseteq \pi_0(R)/\mathfrak{p}$,
  either $I/\mathfrak{p}$ or $\pi_0(R)/I$ has dimension at most $k$--which we can show by localizing at $\mathfrak{p}$.
\end{proof}

\section{Descendability of Faithfully Flat covers of Perfect Stacks} \label{section2}
In this section we apply the results of the previous section to deduce some results about descendability of maps of stacks.

\begin{lemma} \label{cohdimandinjdim}
  Let $\mathscr{C}$ be a small idempotent-complete stable category with a right-bounded t-structure
  such that its $\lambda$-dimension is at most $n$.
  Suppose $F : \opposite{\mathscr{C}} \to \Sp$ is a functor (between stable categories)
  which preserves coconnective objects and has cohomological dimension at most $m$, i.e.
  $F(x) \in \Sp^{[0,m]} $ for all $x \in \opposite{(\mathscr{C}^{\heart})}$.
  Then, $F$ has injective dimension at most $n+m$.
\end{lemma}
\begin{proof}
  We induct on $m$, with the base case being true by definition. So assume $m > 0$.
  Let $I$ be an injective hull of $F$ in $\Ind{\mathscr{C}}$.
  The cofibre, as an object of $\mathrm{Fun}^{ex}(\mathscr{C}^{\operatorname{op}},\Sp)$, preserves coconnective objects
  (because it is coconnective in $\Ind{\mathscr{C}}$ as injective hull are injective on $\pi_0$) and has cohomological dimension at most $m-1$.
  So we are done by the induction hypothesis.
\end{proof}

\begin{proposition}
  Let $\mathscr{C}$ be a Grothendieck prestable category and $T$ be a colimit-preserving monad on $\Sp(\mathscr{C})$.
  Then $\Mod_T(\Sp(\mathscr{C}))$ has a natural t-structure such that
  \begin{itemize}
  \item The connective objects $\Mod_T(\Sp(\mathscr{C}))_{\ge 0}$ are generated under colimits and extensions by
    the objects $F_T(L)$ for $L \in \mathscr{C}$. Here $F_T : \Sp(\mathscr{C}) \to \Mod_T(\Sp(\mathscr{C}))$
    is the left adjoint in the monadic adjunction associated to the monad $T$.
  \item An object of $\Mod_T(\Sp(\mathscr{C}))$ is coconnective if and only if its underlying object in $\Sp(\mathscr{C})$ is.
  \item The connective objects form a Grothendieck prestable category.
  \item The t-structure is right t-complete.
  \end{itemize}
\end{proposition}
\begin{proof}
  \cite{HA} Proposition 1.4.4.11 guarantees the existence of a t-structure on $\Mod_T(\Sp(\mathscr{C}))$
  with the connective part as described. The characterization of the coconnective objects follows directly from
  the fact that coconnective objects are exactly right orthogonal to connective objects shifted by $1$ (\cite{HA} Remark 1.2.1.3).
  It suffices to show that this t-structure is compatible with filtered colimits and is right-complete.

  The coconnective objects in $\Mod_T(\Sp(\mathscr{C}))$
  are closed under filtered colimits as filtered colimits are preserved by the forgetful functor (\cite{HA} Corollary 4.2.3.5)
  and coconnectiveness is reflected by the forgetful functor. The t-structure is right-separated as the forgetful functor is conservative
  (in addition to preserving coconnective objects) and the t-structure on $\Sp(\mathscr{C})$ is right-separated. Coconnective objects
  being closed under filtered colimits (and therefore direct sums) together with right-separatedness implies right-completeness by \cite{HA} Proposition 1.2.1.19.
\end{proof}

We recall a definition of \cite{mathewmondal} (Definition 2.4 in \textit{loc. cit.}),
rephrased as a generalization of the notion of a faithfully flat monad in \cite{SAG} Section D.6.4.
See also \cite{burklundlevy} Definition 4.8.
\begin{definition} \label{fflatdef}
  Let $\mathscr{C}$ be a Grothendieck prestable category. 
  A colimit-preserving monad $T$ on $\Sp(\mathscr{C})$ is called faithfully flat if the functor
  \[F_T : \Sp(\mathscr{C}) \to \Mod_T(\Sp(\mathscr{C}))\]
  is t-exact and conservative on $\mathscr{C}^{\heart}$--where the t-structure on $\Mod_T(\Sp(\mathscr{C}))$ is the one described above.
\end{definition}
\begin{remark} \label{comparisonremark}
  In Definition \ref{fflatdef}, it suffices to check that coconnective objects are preserved because connective objects are always preserved by $F_T$.
  Also, if $T$ is faithfully flat, then $F_T$ is conservative on $\Sp(\mathscr{C})^{\ge 0}$ as $F_T$ is t-exact
  and the t-structure on $\Sp(\mathscr{C})$ is right-separated.
  Finally, we observe that if $T$ is faithfully flat and the t-structure on $\Sp(\mathscr{C})$ is left-separated, then $F_T$ is conservative.
\end{remark}

The following proposition is Proposition 2.5 in \cite{mathewmondal}. We repeat the proof here.
\begin{proposition} \label{equivcondfflat}
  Let $\mathscr{C}$ be a Grothendieck prestable category.
  A colimit-preserving monad $T$ on $\Sp(\mathscr{C})$ is a faithfully flat 
  if and only if for any coconnective object $M \in \Sp(\mathscr{C})^{\ge 0}$, 
  \[\mathrm{cofib}(M \to T(M)) \in \Sp(\mathscr{C})^{\ge 0}\]
\end{proposition}
\begin{proof}
  Suppose $T$ is faithfully flat and $M$ is coconnective. Then it suffices to show that
  \[\pi_1(\mathrm{cofib}(M \to T(M))) \cong 0\]
  As $F_T$ is t-exact and conservative on $\mathscr{C}^{\heart}$, this is equivalent to
  \[\pi_1(\mathrm{cofib}(F_T(M) \to F_T(T(M)))) \cong 0\]
  But the map $F_T \to F_T \circ T$ has a retract, so the cofibre is a summand of $F_T(T(M))$.
  As $F_T(T(M))$ is coconnective, it has no $\pi_1$, and neither do its summands.

  Now suppose 
  \[\mathrm{cofib}(M \to T(M)) \in \Sp(\mathscr{C})^{\ge 0}\]
  for all $M$ coconnective. Then, $T(M)$ is coconnective for all $M$ coconnective--so $F_T$ is t-exact by Remark \ref{comparisonremark} above.
  It suffices to show that $F_T$ is conservative on $\mathscr{C}^{\heart}$.
  But if $M \in \mathscr{C}^{\heart}$ with $F_T(M) \cong 0$, then by assumption
  \[\mathrm{cofib}(M \to T(M)) \cong M[1] \in \Sp(\mathscr{C})^{\ge 0}\]
  which implies $M \cong 0$.
\end{proof}

\begin{proposition} \label{coverfflat}
  Suppose $\mathfrak{X}$ is quasi-geometric stack with a faithfully flat map $p: \Spec{R} \to \mathfrak{X}$.
  Then $p_*p^*$ is a colimit preserving faithfully flat monad on $\QC{\mathfrak{X}} \cong \Sp(\QC{\mathfrak{X}}_{\ge 0})$ (see
  \cite{SAG} Proposition 9.1.3.1).
\end{proposition}
\begin{proof}
  The adjunction $p^* \dashv p_*$ is monadic for a colimit-preserving monad by \cite{SAG} Proposition 6.3.4.6 (as $\mathfrak{X}$
  has quasi-affine diagonal). The faithfully flatness follows \cite{SAG} Remark 9.1.3.4. Note that $p_*p^*$ is $\QC{\mathfrak{X}}$-linear
  by \cite{SAG} Corollary 6.3.4.3.
\end{proof}

We recall the notion of index of descendability from \cite{bhattscholze} Definition 11.18.
\begin{definition}
  Let $\mathscr{C}$ be a symmetric monoidal stable category.
  We say an algebra $A \in \mathrm{CAlg}(C)$ is descendable of index at most $n$
  if the natural map $\fib(\eta_A)^{\otimes n} \to \pmb{1}_{\mathscr{C}}$ is nilhomotopic, where $\eta_A$ is the unit map for $A$.
\end{definition}

By \cite{bhattscholze} Lemma 11.20, the index of descendability is finite iff the algebra is descendable
in the sense of \cite{descendable} Definition 3.18. It may be desirable to modify to index of descendability
down by $1$, so that the unit $\pmb{1}_{\mathscr{C}}$ would be descendable of index $0$ instead of $1$ in general. However, we refrain from
this change here to avoid confusion.

\begin{remark}
  It is possible to define descendability for an exact monad on a stable category in an analogous manner, further generalizing
  the generalization in \cite{SAG} Definition D.3.1.1. However, we do not pursue this further in the current paper.
\end{remark}

\begin{theorem} \label{mainthm1}
  Let $\mathfrak{X}$ be a perfect stack with a faithfully flat map from 
  a Noetherian $\mathbb{E}_{\infty}$-ring of Krull dimension $n$.
  Then every faithfully flat algebra\footnote{we say an algebra $A \in \CAlg(\QC{\mathfrak{X}})$ is faithfully flat if the monad $A \otimes \_$ is}
  in $\QC{\mathfrak{X}}$ is descendable of index at most $cd(\mathfrak{X})+n+1$.
\end{theorem}
\begin{proof}
  Suppose $p: \Spec{R} \to \mathfrak{X}$ is a faithfully flat map where $R$ is a Noetherian $\mathbb{E}_{\infty}$-ring of Krull dimension $n$.
  Let $\Coh{\mathfrak{X}}$ be the full subcategory of $\QC{\mathfrak{X}}$ consisting of almost perfect quasicoherent sheaves on $\mathfrak{X}$ with bounded homotopy groups.
  Because the pullback along $p$ is t-exact and conservative by \cite{SAG} Remark 9.1.3.4, $\Coh{\mathfrak{X}}$ is also the preimage of $\Coh{R}$
  under the functor $p^* : \QC{\mathfrak{X}} \to \QC{\Spec{R}}$.
  The dimension of support filtration on $\Coh{R}$ induces one on $\Coh{\mathfrak{X}}$ of length $n+1$.
  The $(k+1)$-st subquotient of this filtration $\Coh{\mathfrak{X}}_{(k+1)}/\Coh{\mathfrak{X}}_{(k)}$ admits a t-exact and conservative functor
  to $\Coh{R}_{(k+1)}/\Coh{R}_{(k)}$, which has an Artinian heart by Lemma \ref{tstructurelemma} and Theorem \ref{Gabrieltheorem}.
  Hence $\Coh{\mathfrak{X}}_{(k+1)}/\Coh{\mathfrak{X}}_{(k)}$ also has an Artinian heart.
  As it has a bounded t-structure, it is idempotent complete by \cite{SAG} Remark C.6.7.4.
  So we deduce the $\lambda$-dimension of $(\opposite{\Coh{\mathfrak{X}}})$ is at most $n$ by Propositions \ref{lambdabound1} and \ref{lambdabound2}.

  Let
  \begin{equation}\label{qcohembedding}
    i: \QC{\mathfrak{X}} \to \mathrm{Fun}^{ex}(\Coh{\mathfrak{X}},\Sp)
  \end{equation}
    
  be the functor $F \mapsto \Gamma(F \otimes \_)$. Restricted to $\Perf{\mathfrak{X}}$, we have the isomorphisms
  \[i(K) \cong \Hom(K^{\vee},\_) \cong \colim_k\Hom(\tau_{\le k}K^{\vee},\_)\]
  because each homotopy group of the colimit is eventually constant. If $K$ is perfect, one can see that the natural map
  (the isomorphism below follows from the $\Sp$-enriched Yoneda lemma, see \cite{hinich} 6.2.7.)
  \[\Hom(K,F) \to \Hom(i(K),i(F)) \to \Hom(\Hom(\tau_{\le k}K^{\vee},\_),i(F)) \cong \Gamma(F \otimes \tau_{\le k}K^{\vee})\]
  is given by the composition
  \[\Gamma(F \otimes K^{\vee}) \to \Gamma(F \otimes K^{\vee} \otimes K \otimes \tau_{\le k}K^{\vee}) \to \Gamma(F \otimes \tau_{\le k}K^{\vee})\]
  using the unit/counit relations of duality data of $K$, where the map $\mathcal{O} \to K \otimes \tau_{\le k}K^{\vee}$ is the unit map
  followed by truncation in the second variable. Hence the entire map is simply induced by truncation in the second variable
  \[\Gamma(F \otimes K^{\vee}) \to \Gamma(F \otimes \tau_{\le k}K^{\vee})\]
  If $F$ is flat, we have the isomorphism (since $\QC{\mathfrak{X}}$ is left-complete)
  \[\Gamma(F \otimes K^{\vee}) \cong \lim_k \Gamma(F \otimes \tau_{\le k}K^{\vee})\]
  which implies that
  \[\Hom(K,F) \cong \Hom(i(K),i(F))\]
  for all $K$ perfect and $F$ flat. The perfectness assumption on $K$ can then be removed as $i$ preserves colimits and $\QC{\mathfrak{X}}$
  is generated by perfect complexes.

  Now suppose $A$ is a faithfully flat algebra in $\QC{\mathfrak{X}}$ and let $\eta_A$ denote the unit map $\mathcal{O}_{\mathfrak{X}} \to A$.
  Then $\fib(\eta_A) \otimes M \in \QC{\mathfrak{X}}^{\ge 1}$ for all $M \in \QC{\mathfrak{X}}^{\ge 0}$ by Proposition \ref{equivcondfflat}
  (\cite{SAG} Proposition 9.1.3.1 shows that $\QC{\mathfrak{X}}$ is the stablization of the Grothendieck prestable category $\QC{\mathfrak{X}}_{\ge 0}$).
  Hence, $i(\fib(\eta_A)^{\otimes k}) \in \mathrm{Fun}^{ex}(\Coh{\mathfrak{X}},\Sp)^{\ge k}$.

  Now we have
  \begin{equation}
    \Hom_{\QC{\mathfrak{X}}}(\fib(\eta_A)^{\otimes k},\mathcal{O}_{\mathfrak{X}})
    \cong \Hom_{\mathrm{Fun}^{ex}(\Coh{\mathfrak{X}},\Sp)}(i(\fib(\eta_A)^{\otimes k}),i(\mathcal{O})) \in \Sp_{\ge k-n-cd(\mathfrak{X})}
  \end{equation}
  since the injective dimension of $i(\mathcal{O}_{\mathfrak{X}})$ is at most $cd(\mathfrak{X})+n$ by Lemma \ref{cohdimandinjdim}.
  Hence the above Hom space $\Hom_{\QC{\mathfrak{X}}}(\fib(\eta_A)^{\otimes k},\mathcal{O}_{\mathfrak{X}})$ is connected when $k > cd(\mathfrak{X})+n$, as desired.
\end{proof}

\begin{theorem} \label{mainthm1.5}
  Let $\mathfrak{X}$ be a quasi-geometric stack with a faithfully flat map from 
  a discrete regular Noetherian commutative ring of Krull dimension $n$.
  Then every faithfully flat algebra in $\QC{\mathfrak{X}}$ is descendable of index at most $cd(\mathfrak{X})+n+1$.
\end{theorem}
\begin{proof}
  If $\mathfrak{X}$ is a classical quasi-geometric stack, the structure sheaf $\mathcal{O}_{\mathfrak{X}}$
  lies in the heart of $\QC{\mathfrak{X}}$. Therefore, for any faithfully flat algebra $A$ with unit map $\eta_A : \mathcal{O}_{\mathfrak{X}} \to A$,
  we have $\fib(\eta_A)^{\otimes k} \cong \fib(\eta_A)^{\otimes k} \otimes \mathcal{O}_{\mathfrak{X}} \in \QC{\mathfrak{X}}^{\ge k}$.
  Hence, it suffices to show that $\mathcal{O}_{\mathfrak{X}}$ has injective dimension $\le cd(\mathfrak{X})+n$ in $\QC{\mathfrak{X}}$.

  By \cite{SAG} Proposition 9.5.2.3, we know that $\QC{\mathfrak{X}}^{\heart}$ is a locally Noetherian Grothendieck abelian category
  with Noetherian objects given by $\Coh{\mathfrak{X}}^{\heart}$.
  Hence, any $X \in \QC{\mathfrak{X}}^{\heart}$ can be written as an increasing
  union indexed by ordinals as in Proposition \ref{lambdabound2}, where the subquotients are Noetherian objects. Hence, by the argument
  in Proposition \ref{lambdabound2}, we can deduce the injective dimension of $\mathcal{O}_{\mathfrak{X}}$ is bounded by the highest nonvanishing
  cohomology group of $\Hom_{\QC{\mathfrak{X}}}(x,\mathcal{O}_{\mathfrak{X}})$, where $x \in \Coh{\mathfrak{X}}^{\heart}$. However, as $x$ is perfect with
  $x^{\vee} \in \QC{\mathfrak{X}}^{[0,n]}$ by regularity of the faithfully flat cover--we have $\Hom_{\QC{\mathfrak{X}}}(x,\mathcal{O}_{\mathfrak{X}}) \in \Sp^{[0,cd(\mathfrak{X})+n]}$,
  as desired.

\end{proof}
\begin{remark}
  The argument above is similar to the one given in Proposition 5.2.5 of \cite{heyermann},
  who proved this result in the case where $\mathfrak{X}$ is the classifying stack of a profinite group (thought of as an affine group scheme).
\end{remark}

\begin{corollary} \label{autoperfect}
  Let $\mathfrak{X}$ be a quasi-geometric stack of finite cohomological dimension with a faithfully flat map from 
  a discrete regular Noetherian commutative ring of Krull dimension $n$. Then $\mathfrak{X}$ is a perfect stack.
\end{corollary}
\begin{proof}
  Any quasi-geometric stack with finite cohomological dimension has a compact structure sheaf,
  by \cite{SAG} Proposition 9.1.5.3. Therfore, it suffices to check that the perfect
  complexes (= compact objects by the same proposition in \textit{loc. cit.}) generate $\QC{\mathfrak{X}}$.
  Let $p: \Spec{R} \to \mathfrak{X}$ be a faithfully flat map from a regular Noetherian ring of dimension $n$. By Proposition
  \ref{coverfflat}, $p_*\mathcal{O}_{\Spec{R}}$ is a faithfully flat algebra in $\QC{\mathfrak{X}}$, therefore it is descendable by Theorem \ref{mainthm1.5}.
  Therefore, as $\mathcal{O}_{\mathfrak{X}}$ is a retract of a finite colimit of objects of the form $p_*\mathcal{O}_{\Spec{R}} \otimes \mathcal{F}$ by
  \cite{descendable} Proposition 3.20, the same is also true for any object in $\QC{\mathfrak{X}}$. Hence, it suffices to show that
  any object of the form $p_*\mathcal{O}_{\Spec{R}} \otimes \mathcal{F} \cong p_*p^*\mathcal{F}$ (by \cite{SAG} Corollary 6.3.4.3)
  is in $\Ind{\Perf{\mathfrak{X}}}$.

  Also by \cite{SAG} Corollary 6.3.4.3, $p_*$ preserves colimits. As $\QC{\Spec{R}}$ is generated by the structure sheaf
  by colimits and shifts, it suffices to check that $p_*{\mathcal{O}_{\Spec{R}}}$ is in $\Ind{\Perf{\mathfrak{X}}}$.
  The inclusion functor $\Ind{\Perf{\mathfrak{X}}} \to \QC{\mathfrak{X}}$ is a t-exact functor between stablizations
  of Grothendieck prestable categories (\cite{SAG} Example C.1.4.4 and Proposition 9.1.3.1) which is an isomorphism on the heart
  by \cite{SAG} Proposition 9.5.2.3. Hence $\Ind{\Perf{\mathfrak{X}}}$ contains all coconnective objects of $\QC{\mathfrak{X}}$,
  and therefore $p_*\mathcal{O}_{\Spec{R}}$.
\end{proof}

\begin{definition}
  Let $\kappa$ be a regular cardinal. A connective $\mathbb{E}_1$-ring $R$ is called left $\kappa$-Artinian if for all $m \ge 0$,
  $\pi_m(R)$ is an $\kappa$-Artinian object in $\Mod_R^{\heart} \cong \Mod_{\pi_0(R)}^{\heart}$.
\end{definition}

\begin{theorem} \label{mainthm2}
  Let $\mathfrak{X}$ be a perfect stack with a faithfully flat map $f: \Spec{R} \to \mathfrak{X}$.
  If $R$ is $\aleph_n$-Artinian,
  then every faithfully flat algebra in $\QC{\mathfrak{X}}$ is descendable with index at most $cd(\mathfrak{X})+n+1$.
\end{theorem}
\begin{proof}
  The proof is the same as for Theorem \ref{mainthm1}, except
  we replace the category of bounded almost perfect quasicoherent sheaves, $\Coh{\mathfrak{X}}$,
  with the category of bounded quasicoherent sheaves
  whose pullback to $\Spec{R}$ has $\aleph_n$-Artinian cohomology sheaves, $(f^*)^{-1}[(\Mod_R)^b_{\aleph_n\mathrm{-Artin}}]$.
  Because $\aleph_n$-Artinian objects are closed under extensions (Corollary \ref{kappanoetherianext}), we see that truncations of perfect complexes
  are in this subcategory (because the collections of objects in $\Mod_R$ with $\aleph_n$-Artinian cohomology is stable and contains $R$
  by assumption).
  
  By Proposition \ref{lambdabound2}, $\left((f^*)^{-1}[(\Mod_R)^b_{\aleph_n\mathrm{-Artin}}]\right)^{\operatorname{op}}$ has $\lambda$-dimension at most $n$.
  The rest of the proof is identical to Theorem \ref{mainthm1}.
  
\end{proof}
\begin{remark} \label{cardinalremark}
  Unfortunately, Theorem \ref{mainthm2} is not entirely satisfactory.
  Namely, if $\mathfrak{X} = \Spec{R}$ is an affine scheme with $|\pi_*(R)| < \aleph_n$,
  the bound on descendability index based on the proof of \cite{descendable} Proposition 3.32 is $n$.
  However, not every $\mathbb{E}_{\infty}$-ring $R$ with $|\pi_*(R)| < \aleph_n$ is $\aleph_n$-Artinian.
  Nevertheless we can deduce that $R$ is $\aleph_n$-Artinian if $2^{|\pi_*(R)|}<\aleph_n$.
  This means that assuming the Generalized Continuum Hypothesis the gap between our bound and the best known bound for rings is not
  too large, while in some other models of ZFC our bound is close to trivial.

  
\end{remark}

\begin{question}
  In the proofs of Theorems \ref{mainthm1} and \ref{mainthm2}, we made use of a t-exact conservative functor
  $\mathscr{C} \to \mathscr{D}$ to show that the $\lambda$-dimension of $\mathscr{C}$ is bounded above
  by the $\lambda$-dimension of $\mathscr{D}$. We do not know if in general, the existence of a t-exact
  conservative functor between idempotent-complete stable categories with right-bounded t-structures
  implies this inequality of their $\lambda$-dimensions.
\end{question}

\bibliography{references.bib}{}
\bibliographystyle{alpha}
\end{document}